\documentclass[a4paper, 10pt, notitlepage]{article}

\usepackage{amsthm, amsmath, amssymb, latexsym, eufrak}
\usepackage{mathrsfs,cite}
\usepackage[multiple]{footmisc}
\usepackage{graphicx}
\usepackage[ruled,vlined]{algorithm2e}

\theoremstyle{plain}
\newtheorem{theorem}{Theorem}[section]  
\newtheorem{lemma}[theorem]{Lemma}  

\newtheorem{proposition}[theorem]{Proposition}

\usepackage{chngcntr}
\counterwithin{theorem}{section}  

\theoremstyle{definition}

\newtheorem{example}[theorem]{Example}

\theoremstyle{remark}

\DeclareMathOperator{\co}{co}
\DeclareMathOperator{\diag}{diag}

\author{M.V. Dolgopolik\footnote{Institute for Problems in Mechanical Engineering of the Russian Academy of Sciences,
Saint Petersburg, Russia}}
\title{A note on the generalised Hessian of the least squares associated with systems of linear inequalities}

\begin{document}

\maketitle

\begin{abstract}
The goal of this note is to point out an erroneous formula for the generalised Hessian of the least squares associated
with a system of linear inequalities, that was given in the paper ``A finite Newton method for classification'' by O.L.
Mangasarian (Optim. Methods Softw. 17: 913--929, 2002) and reproduced multiple times in other publications. We also
provide sufficient contiditions for the validity of Mangasarian's formula and show that Slater's condition 
guarantees that some particular elements from the set defined by Mangasarian belong to the generalised Hessian of 
the corresponding function.
\end{abstract}

\section{Introduction}

The generalised Jacobian of a locally Lipschitz continuous function $F \colon \mathbb{R}^n \to \mathbb{R}^m$ is defined
as
\[
  \partial F(x) = \co\Big\{ A \in \mathbb{R}^{m \times n} \Bigm| 
  \exists \{ x_n \} \subset D_F \colon \lim_{n \to \infty} x_n = x,
  \: \lim_{n \to \infty} J F(x_n) = A \Big\},
\]
where $D_F$ is the set of points at which $F$ is differentiable and $JF(x_n)$ is the classical Jacobian of $F$ at $x_n$ 
(see \cite{Clarke} for more details). The generalised Jacobian is a nonempty compact convex set. In turn, the
generalised Hessian \cite{HiriartUrruty} of a continuously differentiable function 
$f \colon \mathbb{R}^n \to \mathbb{R}$ with locally Lipschitz continuous gradient is defined as the generalised Jacobian
of the gradient of $f$ (see \cite{HiriartUrruty}) and is denoted by $\partial^2 f(x)$. The generalised 
Hessian is a nonempty compact convex set of symmetric matrices.

In \cite{Mangasarian2002}, the generalised Hessian of the function $f(x) = 0.5 \| (Ax - b)_+ \|^2$ associated with the
system of linear inequalities
\begin{equation} \label{eq:LinearSyst}
  A x \le b
\end{equation}
with some $A \in \mathbb{R}^{m \times n}$ and $b \in \mathbb{R}^m$ was considered. Here $\| \cdot \|$ is the Euclidean
norm and $(x)_+$ is a vector with components $\max\{ 0, x_i \}$. In \cite[Lemma~4]{Mangasarian2002}, it was claimed
that 
\begin{equation} \label{eq:MangasarianFormula}
  \partial^2 f(x) = A^T \diag((Ax - b)_*) A \quad \forall x \in \mathbb{R}^n,
\end{equation}
where
\[
  [(y)_*]_i = \begin{cases}
    1, & \text{if } y_i > 0,
    \\
    [0, 1], & \text{if } y_i = 0,
    \\
    0, & \text{if } y_i < 0,
  \end{cases}
  \quad \forall y \in \mathbb{R}^m,
\]
without any assumptions on the matrix $A$. Formula \eqref{eq:MangasarianFormula} has been reproduced and used in
multiple publications (see, e.g. \cite{FungMangasarian,Mangasarian2004,Keerthi,Mangasarian2006,Tanveer,Balasundaram}).
In many other articles (see, e.g. \cite{Evtushenko2004,Evtushenko2005,Evtushenko2008,Yuan,HoLin}), the following
corollary to formula \eqref{eq:MangasarianFormula} was used without any assumptions on the matrix $A$:
\begin{equation} \label{eq:EvtushenkoFormula}
  A^T D_{\pm}(x) A \in \partial^2 f(x)	\quad \forall x \in \mathbb{R}^n,
\end{equation}
where $D_+(x)$ and $D_-(x)$ are diagonal matrices such that
\[
  D_+(x)_{ii} = \begin{cases}
    1, & \text{if } \langle A_i, x \rangle > b_i,
    \\
    0, & \text{if } \langle A_i, x \rangle \le b_i,
  \end{cases}
  \quad
  D_-(x)_{ii} = \begin{cases}
    1, & \text{if } \langle A_i, x \rangle \ge b_i,
    \\
    0, & \text{if } \langle A_i, x \rangle < b_i
  \end{cases}
\]
for all $i \in \{ 1, \ldots, m \}$.

The aim of this note is to show that neither equality \eqref{eq:MangasarianFormula} nor inclusion
\eqref{eq:EvtushenkoFormula} hold true in the general case and provide sufficient conditions for these formulas to be
valid.

\section{A counterexample}

Let us provide a simple counterexample to \eqref{eq:MangasarianFormula} and \eqref{eq:EvtushenkoFormula}. Let 
$n = m = 2$. Consider the following system of linear inequalities:
\[
  x_1 \le 0, \quad - x_1 \le 0.
\]
In this case $A = \left( \begin{smallmatrix} 1 & 0 \\ -1 & 0 \end{smallmatrix} \right)$, 
$b = \left( \begin{smallmatrix} 0 \\ 0 \end{smallmatrix} \right)$. Hence
\[
  f(x) = \frac{1}{2} \| (Ax - b)_+ \|^2 = \frac{1}{2} \big( \max\{ x_1, 0 \}^2 + \max\{ - x_1, 0 \} \big)^2 
  = \frac{1}{2} x_1^2,
\]
which implies that
\[
  \nabla f(x) = \begin{pmatrix} x_1 \\ 0 \end{pmatrix}, \quad 
  \nabla^2 f(x) = \begin{pmatrix} 1 & 0 \\ 0 & 0 \end{pmatrix} \quad \forall x \in \mathbb{R}^2,
\]
that is, the function $f$ is twice continuously differentiable. Therefore, in particular, 
$\partial^2 f(x) = \nabla^2 f(x)$ for all $x \in \mathbb{R}^2$.

One the other hand, for any $x \in \mathbb{R}^2$ such that $x_1 = 0$ one has
\begin{align*}
  A^T \diag((Ax - b)_*) A = \begin{pmatrix} 1 & -1 \\ 0 & 0 \end{pmatrix} 
  \begin{pmatrix} [0, 1] & 0 \\ 0 & [0, 1] \end{pmatrix}
  \begin{pmatrix} 1 & 0 \\ -1 & 0 \end{pmatrix} 
  &= \begin{pmatrix} [0, 2] & 0 \\ 0 & 0 \end{pmatrix} 
  \\ 
  &\ne \nabla^2 f(x)
\end{align*}
\begin{align*}
  &A^T D_+(x) A = \begin{pmatrix} 1 & -1 \\ 0 & 0 \end{pmatrix} \begin{pmatrix} 0 & 0 \\ 0 & 0 \end{pmatrix}
  \begin{pmatrix} 1 & 0 \\ -1 & 0 \end{pmatrix} = \begin{pmatrix} 0 & 0 \\ 0 & 0 \end{pmatrix} \ne \nabla^2 f(x),
  \\
  &A^T D_-(x) A = \begin{pmatrix} 1 & -1 \\ 0 & 0 \end{pmatrix} \begin{pmatrix} 1 & 0 \\ 0 & 1 \end{pmatrix}
  \begin{pmatrix} 1 & 0 \\ -1 & 0 \end{pmatrix} = \begin{pmatrix} 2 & 0 \\ 0 & 0 \end{pmatrix} \ne \nabla^2 f(x).
\end{align*}
Thus, formulas \eqref{eq:MangasarianFormula} and \eqref{eq:EvtushenkoFormula} do not hold true in the general case.

\section{Sufficient conditions}

Let us provide sufficient conditions for \eqref{eq:MangasarianFormula} and \eqref{eq:EvtushenkoFormula} to be valid.
Denote 
\begin{align*}
  I_0(x) &= \Big\{ i \in \{ 1, \ldots, m \} \Bigm| \langle A_i, x \rangle = b_i \Big\}, 
  \\
  I_+(x) &= \Big\{ i \in \{ 1, \ldots, m \} \Bigm| \langle A_i, x \rangle > b_i \Big\}, 
  \\
  I_-(x) &= \Big\{ i \in \{ 1, \ldots, m \} \Bigm| \langle A_i, x \rangle < b_i \Big\},
\end{align*}
where $A_i$ are the rows of the matrix $A$ written as column vectors, that is, $A^T = (A_1, \ldots, A_m)$.

As usual, we say that Slater's condition holds true, if there exists $\widehat{x} \in \mathbb{R}^n$ such that 
$A \widehat{x} < b$ (i.e. $\langle A_i, x \rangle < b_i$ for all $i \in \{ 1, \ldots, m \}$). As is well-known and easy
to check, Slater's conditions is equivalent to the assumption that the solution set of the system of linear inequalities
\eqref{eq:LinearSyst} has nonempty interior.

\begin{proposition} \label{prp:EvtushenkoFormula}
Let Slater's condition hold true. Then for any $x \in \mathbb{R}^n$ one has $A^T D_{\pm}(x) A \in \partial^2 f(x)$.
\end{proposition}

\begin{proof}
If $I_0(x) = \emptyset$, then
\[
  f(y) = \frac{1}{2} \sum_{i \in I_+(x)} \big( \langle A_i, y \rangle - b_i \big)^2
\]
for any $y$ in a sufficiently small neighbourhood of $x$. Consequenlty, $f$ is twice continuously differentiable at $x$
and the claim of the proposition can be readily verified directly.

Suppose now that $I_0(x) \ne \emptyset$. Choose any sequence $\{ t_n \} \subset (0, + \infty)$ converging to zero and
denote
$y_n = (1 - t_n) x + t_n \widehat{x}$, where $\widehat{x}$ is from Slater's condition. Then 
\[
  \langle A_i, y_n \rangle - b_i 
  = (1 - t_n) (\langle A_i, x \rangle - b_i) + t_n (\langle A_i, \widehat{x} \rangle - b_i).
\]
Therefore, $\langle A_i, y_n \rangle < b_i$ for all $n \in \mathbb{N}$ and $i \in I_0(x) \cup I_-(x)$, while for any 
$i \in I_+(x)$ there exists $n_i \in \mathbb{N}$ such that $\langle A_i, y_n \rangle > b_i$ for any $n \ge n_i$. Hence
for any $n \ge n^* := \max_{i \in I_+(x)} n_i$ there exists a neighbourhood $\mathcal{U}(y_n)$ of $y_n$ such that
\[
  f(y) = \frac{1}{2} \sum_{i \in I_+(x)} \big( \langle A_i, y \rangle - b_i \Big)^2 
  \quad \forall y \in \mathcal{U}(y_n),
\]
which implies that $f$ is twice continuously differentiable at $y_n$ and $\nabla^2 f(y_n) = A^T D_+(x) A$. Passing to
the limit as $n \to \infty$ one obtains that $A^T D_+(x) A \in \partial^2 f(x)$.

Define $z_n = x + t_n(x - \widehat{x})$. Note that
\[
  \langle A_i, z_n \rangle - b_i = \big( \langle A_i, x \rangle - b_i \big) 
  + t_n \Big( \langle A_i, x \rangle - b_i - (\langle A_i, \widehat{x} \rangle - b_i ) \Big)
\]
Hence taking into account the facts that $\langle A_i, \widehat{x} \rangle - b_i < 0$ and the sequence $\{ t_n \}$
converges to zero one can conclude that $\langle A_i, z_n \rangle - b_i > 0$ for any $i \in I_0(x)$ and 
$n \in \mathbb{N}$, and there exists $n^*$ such that 
\[
  \langle A_i, z_n \rangle - b_i \begin{cases}
    > 0, & \text{if } i \in I_+(x), 
    \\
    < 0, & \text{if } i \in I_-(x)
  \end{cases}
  \quad \forall n \ge n^*.
\]
Therefore, for any $n \ge n^*$ there exists a neighbourhood $\mathcal{U}(z_n)$ of $z_n$ such that
\[
  f(y) = \frac{1}{2} \sum_{i \in I_+(x) \cup I_0(x)} \big( \langle A_i, y \rangle - b_i \big)^2 \quad \forall y
\in \mathcal{U}(z_n).
\]
Consequently, $f$ is twice continuously differentiable at $z_n$ for any $n \ge n^*$ and 
$\nabla^2 f(z_n) = A^T D_-(x) A$. Now, passing to the limit as $n \to \infty$ one can conclude that 
$A^T D_-(x) A \in \partial^2 f(x)$.
\end{proof}

As the following example demonstrates, Slater's conditions by itself is not sufficient for equality
\eqref{eq:MangasarianFormula} to hold true.

\begin{example}
Let $n = 2$ and $m = 3$. Consider the following system of linear inequalities:
\begin{equation} \label{eq:SlaterButNotMangasarian}
  x_1 \le 0, \quad x_2 \le 0, \quad x_1 + x_2 \le 0. 
\end{equation}
In this case
\[
  A = \left( \begin{smallmatrix} 1 & 0 \\ 0 & 1 \\ 1 & 1 \end{smallmatrix} \right), \quad b = 0.
\]
and Slater's condition holds true with $\widehat{x} = (-1, -1)^T$. Note that
\begin{align*}
  f(x) &= \frac{1}{2} \Big( \max\{ x_1, 0 \}^2 + \max\{ x_2, 0 \}^2 + \max\{ x_1 + x_2, 0 \}^2 \Big), 
  \\
  \nabla f(x) &= \begin{pmatrix} \max\{ x_1, 0 \} + \max\{ x_1 + x_2, 0 \} 
                  \\
                  \max\{ x_2, 0 \} + \max\{ x_1 + x_2, 0 \}
		\end{pmatrix}
\end{align*}
One can readily verify that $f$ is twice differentiable if and only if $x_1 \ne 0$, $x_2 \ne 0$, and $x_1 + x_2 \ne 0$.
Moreover, one has
\[
  \nabla^2 f(x) = \begin{cases}
    \mathbb{O}_{2 \times 2}, & \text{if } x_1 < 0 \text{ and } x_2 < 0
    \\
    \left( \begin{smallmatrix} 0 & 0 \\ 0 & 1 \end{smallmatrix} \right), 
    & \text{ if } x_1 < 0 \text{ and } x_2 > 0 \text{ and } x_1 + x_2 < 0
    \\
    \left( \begin{smallmatrix} 1 & 1 \\ 1 & 2 \end{smallmatrix} \right), 
  & \text{ if } x_1 < 0 \text{ and } x_2 > 0 \text{ and } x_1 + x_2 > 0
    \\
    \left( \begin{smallmatrix} 1 & 0 \\ 0 & 0 \end{smallmatrix} \right), 
  & \text{ if } x_1 > 0 \text{ and } x_2 < 0 \text{ and } x_1 + x_2 < 0
    \\
    \left( \begin{smallmatrix} 2 & 1 \\ 1 & 1 \end{smallmatrix} \right), 
  & \text{ if } x_1 > 0 \text{ and } x_2 < 0 \text{ and } x_1 + x_2 > 0
    \\
    \left( \begin{smallmatrix} 2 & 1 \\ 1 & 2 \end{smallmatrix} \right), 
  & \text{if } x_1 > 0  \text{ and } x_2 > 0,
  \end{cases}
\] 
where $\mathbb{O}_{2 \times 2}$ is the zero matrix of dimension $2 \times 2$. Therefore
\[
  \partial^2 f(0) = \co\Big\{ \mathbb{O}_{2 \times 2}, 
  \left( \begin{smallmatrix} 0 & 0 \\ 0 & 1 \end{smallmatrix} \right),
  \left( \begin{smallmatrix} 1 & 1 \\ 1 & 2 \end{smallmatrix} \right), 
  \left( \begin{smallmatrix} 1 & 0 \\ 0 & 0 \end{smallmatrix} \right),
  \left( \begin{smallmatrix} 2 & 1 \\ 1 & 1 \end{smallmatrix} \right), 
  \left( \begin{smallmatrix} 2 & 1 \\ 1 & 2 \end{smallmatrix} \right)
  \Big\}.
\]
On the other hand, note that for $x = 0$ one has $(Ax - b)_* = ([0, 1], [0, 1], [0, 1])^T$ and for 
$y = (0, 0, 1)^T \in (Ax - b)_*$ one has
\[
  A^T \diag(y) A = \begin{pmatrix} 1 & 0 & 1 \\ 0 & 1 & 1 \end{pmatrix} 
  \begin{pmatrix} 0 & 0 & 0 \\ 0 & 0 & 0 \\ 0 & 0 & 1 \end{pmatrix}
  \begin{pmatrix} 1 & 0 \\ 0 & 1 \\ 1 & 1 \end{pmatrix} = \begin{pmatrix} 1 & 1 \\ 1 & 1 \end{pmatrix}.
\]
However, $A^T \diag(y) A \notin \partial^2 f(0)$, since otherwise one could find $\alpha_i \ge 0$, 
$i \in \{ 1, \ldots, 6 \}$, such that $\alpha_1 + \ldots + \alpha_6 = 1$ and
\[
  \alpha_2 \begin{pmatrix} 0 & 0 \\ 0 & 1 \end{pmatrix}
  + \alpha_3 \begin{pmatrix} 1 & 1 \\ 1 & 2 \end{pmatrix}
  + \alpha_4 \begin{pmatrix} 1 & 0 \\ 0 & 0 \end{pmatrix}
  + \alpha_5 \begin{pmatrix} 2 & 1 \\ 1 & 1 \end{pmatrix}
  + \alpha_6 \begin{pmatrix} 2 & 1 \\ 1 & 2 \end{pmatrix}
  = \begin{pmatrix} 1 & 1 \\ 1 & 1 \end{pmatrix}
\]
or, equivalently,
\[
  \begin{cases}
    \alpha_3 + \alpha_4 + 2 \alpha_5 + 2 \alpha_6 = 1
    \\
    \alpha_3 + \alpha_5 + \alpha_6 = 1
    \\
    \alpha_2 + 2 \alpha_3 + \alpha_5 + 2 \alpha_6 = 1.
  \end{cases}    
\]
Subtracting the second equation from the first one one gets $\alpha_4 + \alpha_5 + \alpha_6 = 0$, which due to the
nonnegativity of $\alpha_i$ implies that $\alpha_4 = \alpha_5 = \alpha_6 = 0$. Hence with the use of the second equation
we get $\alpha_3 = 1$, and taking into account the third equation we obtain $\alpha_2 + 2 = 1$, which is impossible.
Thus, equality \eqref{eq:MangasarianFormula} does not hold true for the system of linear inequalities
\eqref{eq:SlaterButNotMangasarian}, despite the fact that it satisfies Slater's condition.
\end{example}

Note that in the previous example the vectors $A_i$, $i \in I_0(x)$, are linearly dependent. Our aim is to show that
equality \eqref{eq:MangasarianFormula} holds true, provided these vectors are linearly independent. First we show that
equality \eqref{eq:MangasarianFormula} is satisfied as inclusion ``$\subseteq$'' in the general case and then prove that
the opposite inclusion holds true under the linear independence assumption.

\begin{lemma} \label{lem:MangasarianIncl}
For any $x \in \mathbb{R}^n$ one has $\partial^2 f(x) \subseteq A^T \diag((Ax - b)_*) A$.
\end{lemma}

\begin{proof}
Observe that $\nabla f(x) = (f'_{x_1}(x), \ldots, f'_{x_n}(x))^T$ with
\[
  f'_{x_j}(x) = \sum_{i = 1}^m a_{ij} \max\{ \langle A_i, x \rangle - b_i \rangle, 0 \},
\]
where $a_{ij}$ are the elements of the matrix $A$. As is easily seen, for any locally Lipschitz continuous function 
$F = (F_1, \ldots, F_n)^T$ one has 
\[
  \partial F(\cdot) \subseteq 
  \left( \begin{smallmatrix} \partial F_1(\cdot) \\ \vdots \\ \partial F_n(\cdot) \end{smallmatrix} \right).
\]
Therefore
\[
  \partial^2 f(\cdot) \subseteq  
  \left( \begin{smallmatrix} \partial f'_{x_1}(\cdot)^T \\ \vdots \\ \partial f'_{x_n}(\cdot)^T 
  \end{smallmatrix} \right).
\]
With the use of standard calculus rules for the Clarke subdifferential \cite{Clarke} one gets
\[
  \partial f'_{x_j}(x) \subseteq \sum_{i \in I_+(x)} a_{ij} A_i + \sum_{i \in I_0(x)} a_{ij} \co\{ 0, A_i \}.
\]
Hence taking into account the fact the transposed right-hand side of this inclusion is equal to the $j$-th row of the
matrix $A^T \diag((Ax - b)_*) A$ we arrive at the required result.
\end{proof}

Denote by $|I|$ the cardinality of a set $I$.

\begin{proposition}
Let $|I_0(x)| \le n$ and the vectors $A_i$, $i \in I_0(x)$, be linearly independent for some $x \in \mathbb{R}^n$. Then
$\partial^2 f(x) = A^T \diag((Ax - b)_*) A$.
\end{proposition}

\begin{proof}
If $I_0(x) = \emptyset$, then the claim of the proposition can be readily verified directly (see the proof of
Proposition~\ref{prp:EvtushenkoFormula}). Therefore, suppose that $I_0(x) \ne \emptyset$.

Let $V \subseteq \{ 0, 1 \}^m$ be the set of all those vectors $v$ for which $v_i = 1$ for any $i \in I_+(x)$ and 
$v_i = 0$ for any $i \in I_-(x)$. Note that $V$ is the set of all extreme points of the set $(Ax - b)_*$. Let us show
that $A^T \diag(v) A \in \partial^2 f(x)$ for any $v \in V$. Then thanks to the convexity of the generalised Hessian and
the fact that $(Ax - b)_* = \co V$ one can conclude that
\[
  A^T \diag((Ax - b)_*) A = \co\Big\{ A^T \diag(v) A \Bigm| v \in V \Big\} \subseteq \partial^2 f(x),
\]
which along with Lemma~\ref{lem:MangasarianIncl} implies the required result.

Fix any $v \in V$. From the fact that the vectors $A_i$, $i \in I_0(x)$, are linearly independent it follows that there
exists $y \in \mathbb{R}^n$ such that
\[
  \langle A_i, y \rangle = \begin{cases}
    1, & \text{if } v_i = 1,
    \\
    -1, & \text{if } v_i = 0
  \end{cases}
  \quad \forall i \in I_0(x).
\]
Choose any sequence $\{ t_n \} \subset (0, + \infty)$ converging to zero and denote $x_n = x + t_n y$. Then there exists
$n_0 \in \mathbb{N}$ such that for any $n \ge n_0$ one has
\[
  \langle A_i, x_n \rangle = \langle A_i, x \rangle + t_n \langle A_i, y \rangle  \begin{cases}
    > b_i, & \text{if } i \in I_+(x) \text{ or } (i \in I_0(x) \text{ and } v_i = 1),
    \\
    < b_i, & \text{if } i \in I_-(x) \text{ or } (i \in I_0(x) \text{ and } v_i = 0).
  \end{cases}
\]
Clearly, for any such $n$ there exists a neighbourhood $\mathcal{U}(x_n)$ of $x_n$ such that
\[
  f(y) = \sum_{i \in I_+(x) \cup \{ i \in I_0(x) \colon v_i = 1 \}} (\langle A_i, y \rangle - b_i)^2 
  \quad \forall y \in \mathcal{U}(x_n).
\]
Therefore, for any $n \ge n_0$ the function $f$ is twice continuously differentiable at $x_n$ and, as is easily seen,
$\nabla^2 f(x_n) = A^T \diag(v) A$. Passing to the limit as $n \to \infty$ one gets that 
$A^T \diag(v) A \in \partial^2 f(x)$, which completes the proof.
\end{proof}

\section*{Acknowledgements}

The author is sincerely grateful to professor V.N. Malozemov for drawing the author's attention to paper
\cite{Mangasarian2004}, during a careful examination of which the author made the observations presented in this note.

\bibliographystyle{abbrv}  
\bibliography{Dolgopolik_bibl}

\begin{thebibliography}{10}

\bibitem{Balasundaram}
S.~Balasundaram, D.~Gupta, and S.~C. Prasad.
\newblock A new approach for training {L}agrangian twin support vector machine
  via unconstrained convex minimization.
\newblock {\em Appl. Intelligence}, 46:124--134, 2017.

\bibitem{Clarke}
F.~H. Clarke.
\newblock {\em Optimization and Nonsmooth Analysis}.
\newblock Wiley--Interscience, New York, 1983.

\bibitem{FungMangasarian}
G.~Fung and O.~L. Mangasarian.
\newblock Finite {N}ewton method for {L}agrangian support vector machine
  classification.
\newblock {\em Neurocomputing}, 55:39--55, 2003.

\bibitem{Evtushenko2008}
A.~I. Golikov and {\relax Yu}.~G. Evtushenko.
\newblock Finding the projection of a given point on the set of solutions of a
  linear programming problem.
\newblock {\em Proc. Steklov Inst. Math. (Suppl.)}, 14:68--83, 2008.

\bibitem{Evtushenko2005}
A.~I. Golikov, {\relax Yu}.~G. Evtushenko, and S.~Ketabchi.
\newblock On families of hyperplanes that separate polyhedra.
\newblock {\em Comput. Math. Math. Phys.}, 45:227--–242, 2005.

\bibitem{Evtushenko2004}
A.~I. Golikov, {\relax Yu}.~G. Evtushenko, and N.~Mollaverdi.
\newblock Application of {N}ewton's method for solving large linear programming
  problems.
\newblock {\em Comput. Math. Math. Phys.}, 44:1484--–1493, 2004.

\bibitem{HiriartUrruty}
J.~B. Hiriart-{U}rruty, J.~J. Strodiot, and V.~H. Nguyen.
\newblock Generalized {H}essian matrix and second-order optimality conditions
  for problems with ${C}^{1,1}$ data.
\newblock {\em Appl. Math. Optim.}, 11:43--–56, 1984.

\bibitem{HoLin}
{\relax C.-H}.~Ho and {\relax C.-J}.~Lin.
\newblock Large-scale linear support vector regression.
\newblock {\em J. Machine Learning Res.}, 13:3323--3348, 2012.

\bibitem{Mangasarian2002}
O.~L. Mangasarian.
\newblock A finite {N}ewton method for classification.
\newblock {\em Optim. Methods Softw.}, 17:913--929, 2002.

\bibitem{Mangasarian2004}
O.~L. Mangasarian.
\newblock A {N}ewton method for linear programming.
\newblock {\em J. Optim. Theory Appl.}, 121:1--18, 2004.

\bibitem{Mangasarian2006}
O.~L. Mangasarian.
\newblock Exact 1-norm support vector machines via unconstrained convex
  differentiable minimization.
\newblock {\em J. Machine Learning Res.}, 7:1517--1530, 2006.

\bibitem{Keerthi}
S.~{Sathiya Keerthi} and D.~De{C}oste.
\newblock A modified finite {N}ewton method for fast solution of large scale
  linear {S}{V}{M}s.
\newblock {\em J. Machine Learning Res.}, 6:341--361, 2005.

\bibitem{Tanveer}
M.~Tanveer.
\newblock Newton method for implicit {L}agrangian twin support vector machines.
\newblock {\em Int. J. Machine Learning and Cybernetics}, 6:1029--1040, 2015.

\bibitem{Yuan}
{\relax G.-X}.~Yuan, {\relax K.-W}.~Chang, {\relax C.-J}.~Hsieh, and {\relax
  C.-J}.~Lin.
\newblock A comparison of optimization methods and software for large-scale
  {L}1-regularized linear classification.
\newblock {\em J. Machine Learning Res.}, 11:3183--3234, 2010.

\end{thebibliography}

\end{document}